\documentclass[11pt]{amsart}
\usepackage{amssymb}
\usepackage{mathrsfs}
\usepackage{enumerate}
\makeatletter
\@namedef{subjclassname@2010}{%
  \textup{2010} Mathematics Subject Classification}
\makeatother
\newtheorem{lemma}{Lemma}
\newtheorem{theorem}{Theorem}[section]
\newtheorem{proposition}{Proposition}
\newtheorem{corollary}{Corollary}
\newtheorem{definition}{Definition}
\newtheorem{remark}{Remark}
\newtheorem{example}{Example}

\begin{document}

\title{CONSTRUCTIONS OF $k$-g-FUSION FRAMES AND THEIR DUALS IN HILBERT SPACES}

\author[V. Sadri]{Vahid Sadri}
\address{Institute of Fundamental Sciences\\ University of Tabriz\\ Iran\\ Tel. 00989144112669.}
\email{vahidsadri57@gmail.com}

\author[R. Ahmadi]{Reza Ahmadi}
\address{Institute of Fundamental Sciences\\ University of Tabriz\\ Iran\\}
\email{rahmadi@tabrizu.ac.ir}

\author[A. Rahimi]{Asghar Rahimi}
\address{Department of Mathematics\\ University of Maragheh\\ Iran\\}
\email{rahimi@maragheh.ac.ir}

\begin{abstract}
Frames for operators or $k$-frames were recently considered by
G$\breve{\mbox{a}}$vruta (2012) in connection with atomic systems.
Also generalized frames are important frames in the Hilbert space of
bounded linear operators. Fusion frames, which are a special case of
generalized frames have various applications. This paper introduces
the concept of generalized fusion frames for operators aka
$k$-g-fusion frames and we get some results for characterization of
these frames. We further discuss on duals and Q-duals in connection
with $k$-g-fusion frames. Also we obtain some useful identities for
these frames. We also give several methods to construct $k$-g-fusion
frames. The results of this paper can be used in sampling theory
which are developed by g-frames and especially fusion frames. In the
end, we discuss the stability of a more general perturbation for
$k$-g-fusion frames.
\end{abstract}



\maketitle

\section{Introduction}
Frames in Hilbert spaces were first proposed by Duffin and Schaeffer
in the context of non-harmonic Fourier series \cite{ds}. Now, frames
have been widely applied in signal processing, sampling, filter bank
theory, system modeling, Quantum information, cryptography, etc
(\cite{bhf}, \cite{eldar}, \cite{fer}, \cite{st} and \cite{rens}).
We can say that fusion frames are the generalization of conventional
classical frames and special cases of g-frames in the field of frame
theory. The fusion frames are in fact more susceptible due to
complicated relations between the structure of the sequence of
weighted subspaces and the local frames in the subspaces and due to
the extreme sensitivity with respect to changes of the weights.

 Frames for operators or $k$-frames have been introduced by G$\breve{\mbox{a}}$vruta in \cite{ga} to study the nature of atomic systems for a separable Hilbert space  with respect to a bounded linear operator $k$. It is a well-known fact that $k$-frames are more general than the classical frames and due to higher generality of $k$-frames, many properties of frames may not hold for $k$-frames. Recently, we  presented g-fusion frames in \cite{sad}. This paper presents $k$-g-fusion frames with respect to a bounded linear operator on a separable Hilbert space which are a generalization of g-fusion frames.

Throughout this paper, $H$ and $K$ are separable Hilbert spaces and $\mathcal{B}(H,K)$ is the collection of all bounded linear operators of $H$ into $K$. If $K=H$, then $\mathcal{B}(H,H)$ will be denoted by $\mathcal{B}(H)$. Also, $\pi_{V}$ is the orthogonal projection from $H$ onto a closed subspace $V\subset H$ and  $\lbrace H_j\rbrace_{j\in\Bbb J}$ is a sequence of Hilbert spaces where $\Bbb J$ is a subset of $\Bbb Z$.

For the proof of the following lemma, refer to \cite{ga}.
\begin{lemma}\label{l1}
Let $V\subseteq H$ be a closed subspace, and $T$ be a linear  bounded operator on $H$. Then
$$\pi_{V}T^*=\pi_{V}T^* \pi_{\overline{TV}}.$$
If $T$ is a unitary (i.e. $T^*T=Id_{H}$), then
$$\pi_{\overline{Tv}}T=T\pi_{v}.$$
\end{lemma}
\begin{definition}\textbf{($k$-frame)}.
Let $\{f_j\}_{j\in\Bbb J}$ be a sequence of members of $H$ and
$k\in\mathcal{B}(H)$. We say that $\{f_j\}_{j\in\Bbb J}$ is a
$k$-frame  for $H$ if there exist $0<A\leq B<\infty$ such that for
each $f\in H$,
\begin{eqnarray*}
A\Vert k^*f\Vert^2\leq\sum_{j\in\Bbb J}\vert\langle f,f_j\rangle\vert^2\leq B\Vert f\Vert^2.
\end{eqnarray*}
\end{definition}
\begin{definition}(\textbf{g-fusion frame}).
Let $W=\lbrace W_j\rbrace_{j\in\Bbb J}$ be a collection of closed
subspaces of $H$, $\lbrace v_j\rbrace_{j\in\Bbb J}$ be a family of
weights, i.e. $v_j>0$  and $\Lambda_j\in\mathcal{B}(H,H_j)$ for each
$j\in\Bbb J$. We say $\Lambda:=(W_j, \Lambda_j, v_j)$ is a
generalized fusion frame (or g-fusion frame) for $H$ if there exist
$0<A\leq B<\infty$ such that for each $f\in H$,
\begin{eqnarray}\label{g}
A\Vert f\Vert^2\leq\sum_{j\in\Bbb J}v_j^2\Vert \Lambda_j \pi_{W_j}f\Vert^2\leq B\Vert f\Vert^2.
\end{eqnarray}
\end{definition}
If an operator $ u$ has closed range, then there exists a
right-inverse operator $u^ \dagger$ (pseudo-inverse of $u$) in the
following sences (see \cite{ch}).
\begin{lemma}\label{l3}
Let $u\in\mathcal{B}(K,H)$  be a bounded operator with closed range $\mathcal{R}_{u}$. Then there exists a bounded operator $u^\dagger \in\mathcal{B}(H,K)$ for which
$$uu^{\dagger} x=x, \ \ x\in \mathcal{R}_{u}.$$
\end{lemma}
\begin{lemma}\label{Ru}
Let $u\in\mathcal{B}(K,H)$. Then the following assertions holds:
\begin{enumerate} \item
 $\mathcal{R}_u$ is closed in $H$ if and only if $\mathcal{R}_{u^{\ast}}$ is closed in $K$.
\item $(u^{\ast})^\dagger=(u^\dagger)^\ast$.
\item
The orthogonal projection of $H$ onto $\mathcal{R}_{u}$ is given by $uu^{\dagger}$.
\item
The orthogonal projection of $K$ onto $\mathcal{R}_{u^{\dagger}}$ is given by $u^{\dagger}u$.\item$\mathcal{N}_{{u}^{\dagger}}=\mathcal{R}^{\bot}_{u}$ and $\mathcal{R}_{u^{\dagger}}=\mathcal{N}^{\bot}_{u}$.
 \end{enumerate}
\end{lemma}
\begin{lemma}(\cite{dag}).\label{dag}
Let $L_1\in\mathcal{B}(H_1,H)$ and $L_2\in\mathcal{B}(H_2,H)$ be  on given Hilbert spaces. Then the following assertions are equivalent:
\begin{enumerate}\item
$\mathcal{R}(L_1)\subseteq\mathcal{R}(L_2)$;\item
$L_1L_1^*\leq\lambda^2 L_2L_2^*$ for some$\lambda>0$;\item
there exists a  mapping $u\in\mathcal{B}(H,H_2)$ such that $L_1=L_2 u$.
\end{enumerate}

Moreover, if those condition are valid, then there exists a unique operator $u$ such that
\begin{enumerate}
\item[(a)] $\Vert u\Vert^2=\inf\{\alpha>0 \ \vert \ L_1L_1^*\leq\alpha L_2L_2^*\};$
\item[(b)] $\ker{L_1}=\ker{u};$
\item[(c)] $\mathcal{R}(u)\subseteq\overline{\mathcal{R}(L_2^*)}.$
\end{enumerate}
\end{lemma}
\section{$k$-g- Fusion Frames}
 We define the space $\mathscr{H}_2:=(\sum_{j\in\Bbb J}\oplus H_j)_{\ell_2}$ by
\begin{eqnarray}
\mathscr{H}_2=\big\lbrace \lbrace f_j\rbrace_{j\in\Bbb J} \ : \
f_j\in H_j , \ \sum_{j\in\Bbb J}\Vert f_j\Vert^2<\infty\big\rbrace
\end{eqnarray}
with the inner product defined by
$$\langle \lbrace f_j\rbrace, \lbrace g_j\rbrace\rangle=\sum_{j\in\Bbb J}\langle f_j, g_j\rangle.$$
It is clear that $\mathscr{H}_2$ is a Hilbert space with pointwise operations.
\begin{definition}
Let $W=\lbrace W_j\rbrace_{j\in\Bbb J}$ be a collection of closed
subspaces of $H$, $\lbrace v_j\rbrace_{j\in\Bbb J}$ be a family of
weights, i.e. $v_j>0$, $\Lambda_j\in\mathcal{B}(H,H_j)$ for each
$j\in\Bbb J$ and $k\in\mathcal{B}(H)$. We say $\Lambda:=(W_j,\Lambda_j, v_j)$ is a $k$-g- fusion frame  for $H$ if there exist
$0<A\leq B<\infty$  such that for each $f\in H$,
\begin{eqnarray}\label{g}
A\Vert k^*f\Vert^2\leq\sum_{j\in\Bbb J}v_j^2\Vert \Lambda_j \pi_{W_j}f\Vert^2\leq B\Vert f\Vert^2.
\end{eqnarray}
\end{definition}
When $k=id_H$, we get the g-fusion frame for $H$. Throughout this paper, $\Lambda$ will be a triple $(W_j, \Lambda_j, v_j)$ with $j\in\Bbb J$ unless otherwise noted.
We say $\Lambda$ is a Parseval $k$-g-fusion frame whenever
$$\sum_{j\in\Bbb J}v_j^2\Vert \Lambda_j \pi_{W_j}f\Vert^2=\Vert k^*f\Vert^2.$$
The synthesis and the analysis  operators in the $k$-g-fusion frames are defined by
\begin{align*}
T_{\Lambda}&:\mathscr{H}_2\longrightarrow H \ \ \ \ \ \ \ \ \ , \ \ \ \ \ \ T_{\Lambda}^*:H\longrightarrow\mathscr{H}_2\\
T_{\Lambda}(\lbrace f_j\rbrace_{j\in\Bbb J})&=\sum_{j\in\Bbb J}v_j \pi_{W_j}\Lambda_{j}^{*}f_j \ \ \ , \ \ \ \ T_{\Lambda}^*(f)=\lbrace v_j \Lambda_j \pi_{W_j}f\rbrace_{j\in\Bbb J}.
\end{align*}
Thus, the $k$-g-fusion frame operator is given by
$$S_{\Lambda}f=T_{\Lambda}T^*_{\Lambda}f=\sum_{j\in\Bbb J}v_j^2 \pi_{W_j}\Lambda^*_j \Lambda_j \pi_{W_j}f$$
and
\begin{equation}\label{sf1}
\langle S_{\Lambda}f, f\rangle=\sum_{j\in\Bbb J}v_j^2\Vert \Lambda_j \pi_{W_j}f\Vert^2,
\end{equation}
for all $f\in H$. Therefore,
\begin{equation}\label{sf2}
\langle Akk^*f, f\rangle\leq \langle S_{\Lambda}f, f\rangle\leq \langle B f, f\rangle
\end{equation}
or
\begin{equation}\label{sf3}
Akk^*\leq  S_{\Lambda}\leq B I.
\end{equation}
Hence, we conclude that:
\begin{proposition}\label{pr3}
Let $\Lambda$ be a g-fusion Bessel sequence for $H$. Then $\Lambda$ is a $k$-g-fusion frame for $H$ if and only if there exists $A>0$ such that $S_{\Lambda}\geq Akk^*$.
\end{proposition}
\begin{remark}\label{re1}
In $k$-g-fusion frames, like $k$-frames and $k$-fusion frames, the $k$-g-fusion operator is not invertible. But,
if $k\in\mathcal{B}(H)$ has closed range, then the operator $S_\Lambda$ is an invertible operator on a subspace of $\mathcal{R}(k)\subset H$.
Indeed, suppose that $f\in\mathcal{R}(k)$, then
$$\Vert f\Vert^2=\Vert (k^\dagger\vert_{\mathcal{R}(k)})^*k^*f\Vert^2\leq\Vert k^\dagger\Vert^2 \Vert k^*f\Vert^2.$$
Thus, we have
$$A\Vert k^{\dagger}\Vert^{-2}\Vert f\Vert^2\leq\langle S_{\Lambda}f, f\rangle\leq B\Vert f\Vert^2,$$
which implies that $S_{\Lambda}:\mathcal{R}(k)\rightarrow S_{\Lambda}(\mathcal{R}(k))$ is a homeomorphism, furthermore, for each $f\in S_{\Lambda}(\mathcal{R}(k))$ we have
\begin{equation*}
B^{-1}\Vert f\Vert^2\leq\langle (S_{\Lambda}\vert_{\mathcal{R}(k)})^{-1}f, f\rangle\leq A^{-1}\Vert k^{\dagger}\Vert^2\Vert f\Vert^2.
\end{equation*}
\end{remark}
\begin{remark}\label{re2}
Since $S_{\Lambda}\in\mathcal{B}(H)$ is positive and self-adjoint and $\mathcal{B}(H)$ is a $C^*$-algabra, then $S^{-1}_{\Lambda}$ is positive and self-adjoint too whenever $k\in\mathcal{B}(H)$ has closed range. Now, for each $f\in S_{\Lambda}(\mathcal{R}(k))$ we can write
\begin{align*}
\langle kf, f\rangle&=\langle kf, S_{\Lambda}S_{\Lambda}^{-1}f\rangle\\
&=\langle S_{\Lambda}(kf), S_{\Lambda}^{-1}f\rangle\\
&=\langle\sum_{j\in\Bbb J}v_j^2 \pi_{W_j}\Lambda^*_j \Lambda_j \pi_{W_j}kf, S_{\Lambda}^{-1}f\rangle\\
&=\sum_{j\in\Bbb J}v_j^2 \langle S_{\Lambda}^{-1}\pi_{W_j}\Lambda^*_j \Lambda_j \pi_{W_j}kf, f\rangle.
\end{align*}
\end{remark}
\begin{theorem}
Let $u\in\mathcal{B}(H)$ be an invertible operator on $H$ and $\Lambda$ be a $k$-g-fusion frame for $H$ with bounds $A$ and $B$. Then, $\Gamma:=(uW_j, \Lambda_j \pi_{W_j}u^{*}, v_j)$ is a $uk$-g-fusion frame for $H$.
\end{theorem}
\begin{proof}
Let $f\in H$. Then by applying Lemma \ref{l1}, with $u$ instead of $T$, we have
\begin{align*}
\sum_{j\in\Bbb J}v_j^2\Vert \Lambda_j \pi_{W_j}u^{*}\pi_{uW_j}f\Vert^2&=\sum_{j\in\Bbb J}v_j^2\Vert \Lambda_j \pi_{W_j}u^{*}f\Vert^2\\
&\leq B\Vert u^{*}f\Vert^2\\
&\leq B\Vert u\Vert^2\Vert f\Vert^2.
\end{align*}
So, $\Gamma$ is a g-fusion Bessel sequence for $H$. On the other hand,
\begin{align*}
\sum_{j\in\Bbb J}v_j^2\Vert \Lambda_j \pi_{W_j}u^{*}\pi_{uW_j}f\Vert^2&=\sum_{j\in\Bbb J}v_j^2\Vert \Lambda_j \pi_{W_j}u^{*}f\Vert^2\\
&\geq A\Vert k^*u^{*}f\Vert^2
\end{align*}
and the proof is completed.
\end{proof}
\begin{corollary}
If $u\in\mathcal{B}(H)$ is an invertible operator on Hilbert spaces, $\Lambda$ is a $k$-g-fusion frame for $H$ with bounds $A, B$ and $ku=uk$, then $\Gamma:=(uW_j, \Lambda_j \pi_{W_j}u^{*}, v_j)$ is a $k$-g-fusion frame for $H$.
\end{corollary}
\begin{theorem}
Let $u\in\mathcal{B}(H)$ be an invertible and unitary operator on $H$ and $\Lambda$ be a $k$-g-fusion frame for $H$ with bounds $A$ and $B$. Then, $(uW_j, \Lambda_j u^{-1}, v_j)$ is a $(u^{-1})^*k$-g-fusion frame for $H$.
\end{theorem}
\begin{proof}
Using Lemma \ref{l1},  we can write for any $f\in H$,
\begin{align*}
A\Vert k^*u^{-1}f\Vert^2\leq\sum_{j\in\Bbb J}v_j^2\Vert \Lambda_j u^{-1}\pi_{uW_j}f\Vert^2\leq B\Vert u^{-1}\Vert^2\Vert f\Vert^2.
\end{align*}
\end{proof}
\begin{corollary}
If $u\in\mathcal{B}(H)$ is an invertible and unitary operator on Hilbert spaces, $\Lambda$ is a $k$-g-fusion frame for $H$ with bounds $A, B$ and $k^*u=uk^*$, then $(uW_j, \Lambda_j u^{-1}, v_j)$ is a $k$-g-fusion frame for $H$.
\end{corollary}
\begin{theorem}\label{th2}
If $u\in\mathcal{B}(H)$, $\Lambda$ is a $k$-g-fusion frame for $H$ with bounds $A, B$ and $\mathcal{R}(u)\subseteq\mathcal{R}(k)$, then $\Lambda$ is a $u$-g-fusion frame for $H$.
\end{theorem}
\begin{proof}
Via Lemma \ref{dag}, there exists $\lambda>0$ such that $uu^*\leq\lambda^2kk^*$. Thus, for each $f\in H$ we have
$$\Vert u^*f\Vert^2=\langle uu^*f, f\rangle\leq\lambda^2\langle kk^*f, f\rangle=\lambda^2\Vert k^*f\Vert^2.$$
It follows that
$$\frac{A}{\lambda^2}\Vert u^*f\Vert^2\leq\sum_{j\in\Bbb J}v_j^2\Vert \Lambda_j\pi_{W_j}f\Vert^2.$$
\end{proof}
In the following example is showed that the condition $\mathcal{R}(u)\subseteq\mathcal{R}(k)$ in Corollary \ref{th2} is necessary.
\begin{example}
Let $H=\Bbb R^3$ and $\{e_1, e_2, e_3\}$ be an orthonormal basis for $H$. We define two operators $k$ and $u$ on $H$ by
\begin{align*}
&ke_1=e_2, \ \ \ ke_2=e_3, \ \ \ ke_3=e_3;\\
&ue_1=0, \ \ \ ue_2=e_1, \ \ \ ue_3=e_2.
\end{align*}
Suppose that $W_j=H_j:=\mbox{span}\{e_j\}$ where $j=1,2,3$. Let
\begin{align*}
\Lambda_j f:=\langle f, e_j\rangle e_j=f_j e_j,
\end{align*}
where $f=(f_1, f_2, f_3)$ and $j=1,2,3$. It is clear that $(W_j,
\Lambda_j, 1)$ is a $k$-g-fusion frame for $H$ with bounds
$\frac{1}{2}$ and $1$, respectively. Assume that $(W_j, \Lambda_j,
1)$ is a $u$-g-fusion frame for $H$. Then, by Proposition \ref{pr3},
there exists $A>0$ auch that $S_{\Lambda}=kk^*\geq Auu^*$. So, by
Lemma \ref{dag}, $\mathcal{R}(u)\subseteq \mathcal{R}(k)$. But, this
is a contradiction with $\mathcal{R}(k)\nsubseteq \mathcal{R}(u)$,
since $e_1\in \mathcal{R}(u)$ but $e_1\notin \mathcal{R}(k)$.
\end{example}
\begin{theorem}
Let $\Lambda:=(W_j, \Lambda_j, v_j)$ and $\Theta:=(W_j, \Theta_j, v_j)$ be two g-fusion Bessel sequences for $H$ with bounds $B_1$ and $B_2$, respectively. Suppose that $T_{\Lambda}$ and $T_{\Theta}$ be their analysis operators such that $T_{\Theta}T^*_{\Lambda}=k^*$ where $k\in\mathcal{B}(H)$. Then, both $\Lambda$ and $\Theta$ are $k$-g-fusion frames.
\end{theorem}
\begin{proof}
For each $f\in H$ we have
\begin{align*}
\Vert k^*f\Vert^4&=\langle k^*f, k^*f\rangle^2\\
&=\langle T^*_{\Lambda}f, T^*_{\Theta}k^*f\rangle^2\\
&\leq\Vert T^*_{\Lambda}f\Vert^2 \Vert T^*_{\Theta}k^*f\Vert^2\\
&=\big(\sum_{j\in\Bbb I}v_j^2\Vert \Lambda_j \pi_{W_j}f\Vert^2\big)\big(\sum_{j\in\Bbb I}v_j^2\Vert \Theta_j \pi_{W_j}k^*f\Vert^2\big)\\
&\leq\big(\sum_{j\in\Bbb I}v_j^2\Vert \Lambda_j \pi_{W_j}f\Vert^2\big) B_2 \Vert k^*f\Vert^2,
\end{align*}
thus, $B_2^{-1}\Vert k^*f\Vert^2\leq\sum_{j\in\Bbb I}v_j^2\Vert \Lambda_j \pi_{W_j}f\Vert^2$. This means that $\Lambda$ is a $k$-g-fusion frame for $H$. Similarly, $\Theta$ is a $k$-g-fusion frame with the lower bound $B_1^{-1}$.
\end{proof}

\section{$Q$-Duality of $k$-g-Fusion Frames}
In this section, we shall define duality of $k$-g-fusion frames and present some properties of them.
\begin{definition}
Let $\Lambda=(W_j, \Lambda_j, v_j)$ be a $k$-g-fusion frame for $H$. A g-fusion Bessel sequence $\tilde{\Lambda}:=(\tilde{W_j}, \tilde{\Lambda_j}, \tilde{v_j})$ is called $Q$-dual $k$-g-fusion frame (or brevity $Qk$-gf dual) for $\Lambda$ if there exists a bounded linear operator $Q:\mathscr{H}_2\rightarrow\mathscr{\tilde{H}}_2$ such that
\begin{equation}
T_{\Lambda}Q^*T_{\tilde{\Lambda}}^*=k,
\end{equation}
where $\mathscr{\tilde{H}}_2=(\sum_{j\in\Bbb J}\oplus \tilde{H}_j)_{\ell_2}$.
\end{definition}
Like another $k$-frames, the following presents equivalent conditions of the duality.
\begin{proposition}
Let $\tilde{\Lambda}$ be a $Qk$-gf dual for $\Lambda$. The following conditions are equvalent:
\begin{enumerate}
\item
$T_{\Lambda}Q^*T^*_{\tilde{\Lambda}}=k$;\item
$T_{\tilde{\Lambda}}QT^*_{\Lambda}=k^*$;\item
for each $f, f'\in H$, we have
$$\langle kf, f'\rangle=\langle T^*_{\tilde{\Lambda}}f,QT^*_{\Lambda}f'\rangle=\langle Q^*T^*_{\tilde{\Lambda}}f,T^*_{\Lambda}f'\rangle.$$
\end{enumerate}
\end{proposition}
\begin{proof}
Straightforward.
\end{proof}
\begin{theorem}
If $\tilde{\Lambda}$ is a $Qk$-gf dual for $\Lambda$, then $\tilde{\Lambda}$ is a $k^*$-g-fusion frame for $H$.
\end{theorem}
\begin{proof}
Let $f\in H$ and $B$ be an upper bound of $\Lambda$. Therefore,
\begin{align*}
\Vert kf\Vert^4&=\vert\langle kf, kf\rangle\vert^2\\
&=\vert\langle T_{\Lambda}Q^*T^*_{\tilde{\Lambda}}f, kf\rangle\vert^2\\
&=\vert\langle T^*_{\tilde{\Lambda}}f, QT^*_{\Lambda}kf\rangle\vert^2\\
&\leq\Vert T^*_{\tilde{\Lambda}}f\Vert^2\Vert Q\Vert^2 B\Vert kf\Vert^2\\
&=\Vert Q\Vert^2 B\Vert kf\Vert^2\sum_{j\in\Bbb J}\tilde{v}_j^2\Vert\tilde{\Lambda}_j\pi_{\tilde{W}_j}f\Vert^2
\end{align*}
and by definition, this completes the proof.
\end{proof}
\begin{corollary}
Assume $C_{op}$ and $D_{op}$ are the optimal bounds of $\tilde{\Lambda}$, respectively. Then
\begin{eqnarray*}
C_{op}\geq B_{op}^{-1}\Vert Q\Vert^{-2} \ \ \ and \ \ \ D_{op}\geq A_{op}^{-1}\Vert Q\Vert^{-2}
\end{eqnarray*}
which $A_{op}$ and $B_{op}$ are the optimal bounds of $\Lambda$, respectively.
\end{corollary}
Suppose that $\Lambda$ be a $k$-g-fusion frame for $H$. Since $S_F\geq Akk^*$, then by Lemma \ref{dag}, there exists an opeartor $u\in\mathcal{B}\big(H, \mathscr{H}_2)$ such that
\begin{eqnarray}\label{r8}
T_{\Lambda} u=k.
\end{eqnarray}
Now, we define the $j$-th component of $uf$ by $u_jf=(uf)_j$ for each $f\in H$. It is clear that $u_j\in\mathcal{B}(H, H_j)$. By this operator, we can construct some $Qk$-gf duals for $\Lambda$.
\begin{theorem}
Let $\Lambda$ be a $k$-g-fusion frame for $H$. If $u$ be an operator as in (\ref{r8}) and $\tilde{W_j}:=u^*_juW_j$. Then $\tilde{\Lambda}:=(\tilde{W_j}, \Lambda_j, v_j)$ is a  $Qk$-gf dual for $\Lambda$.
\end{theorem}
\begin{proof}
Define the mapping
\begin{align*}
\Phi_0:\mathcal{R}(T_{\tilde{\Lambda}}^*)&\rightarrow\mathscr{H}_2,\\
\Phi_0(T_{\tilde{\Lambda}}^*f)&=uf.
\end{align*}
Then $\Phi_0$ is well-defined. indeed, if $f_1 ,f_2\in H$ and $T_{\tilde{\Lambda}}^*f_1=T_{\tilde{\Lambda}}^*f_2$, then $\pi_{\tilde{W_j}}(f_1-f_2)=0$. Therefore, for any $j\in\Bbb J$,
$$f_1-f_2\in(\tilde{W_j})^{\perp}=\mathcal{R}(u_j^*)^{\perp}=\ker{u_j}.$$
Thus, $uh_1=uh_2$.
It is clear that $\Phi_0$ be bounded and linear. Therefore, it has a unique linear extension (also denoted $\Phi_0$) to $\overline{\mathcal{R}(T_{\tilde{\Lambda}}^*)}$. Define $\Phi$ on $\mathscr{H}_2$ by setting
\begin{align*}
\Phi=\begin{cases}
\Phi_0,& \mbox{on} \ \overline{\mathcal{R}(T_{\tilde{\Lambda}}^*)},\\
0,& \mbox{on} \ \overline{\mathcal{R}(T_{\tilde{\Lambda}}^*)}^{\perp}
\end{cases}
\end{align*}
and let $Q:=\Phi^*$. This implies that $Q^*\in\mathcal{B}(\mathscr{H}_2, \mathscr{H}_2)$ and
$$T_{\Lambda}Q^*T_{\tilde{\Lambda}}^*=T_{\Lambda}\Phi T_{\tilde{\Lambda}}^*=T_{\Lambda} u=k.$$
\end{proof}
\begin{proposition}
Let $\Lambda$ be a $k$-g-fusion frame with optimal bounds of $A_{op}$ and $B_{op}$, respectively and $k$ has closed range. Then
\begin{eqnarray*}
B_{op}=\Vert S_{\Lambda}\Vert=\Vert T_{\Lambda}\Vert^2 \ \ \ , \ \ \ A_{op}=\Vert u_0\Vert^{-2}
\end{eqnarray*}
which $u_0$ is the unique solotion of equation (\ref{r8}).
\end{proposition}
\begin{proof}
Via Lemma \ref{dag}, the equation (\ref{r8}) has a unique solition as $u_0$ such that
\begin{align*}
\Vert u_0\Vert^2&=\inf\{\alpha>0 \ \vert \ kk^*\leq\alpha T_{\Lambda}T^{*}_{\Lambda}\}\\
&=\inf\{\alpha>0 \ \vert \ \Vert k^* f\Vert^2\leq\alpha \Vert T^*_{\Lambda} f\Vert^2 \ , \ f\in H\}.
\end{align*}
Now, we have
\begin{align*}
A_{op}&=\sup\{A>0 \ \vert \ A\Vert k^*f\Vert^2\leq\Vert T^*_{\Lambda} f\Vert^2 \ , \ f\in H\}\\
&=\Big(\inf\{\alpha>0 \ \vert \ \Vert k^*f\Vert^2\leq\alpha\Vert T^*_{\Lambda} f\Vert^2 \ , \ f\in H\}\Big)^{-1}\\
&=\Vert u_0\Vert^{-2}.
\end{align*}
\end{proof}
\subsection{$k$-g-fusion duals and some identities}
\begin{definition}\label{dual}
Let $\Lambda$ be a $k$-g-fusion frame for $H$. A g-fusion Bessel
sequence $\tilde{\Lambda}=(\tilde{W_j}, \tilde{\Lambda_j}, v_j)$ is
called a $k$-g-fusion dual of $\Lambda$ if for each $f\in H$,
\begin{equation}\label{d1}
kf=\sum_{j\in\Bbb J}v_j^2\pi_{W_j}\Lambda^*_{j}\tilde{\Lambda_j}\pi_{\tilde{W_j}}f.
\end{equation}
\end{definition}
In this case, we can deduce that $\tilde{\Lambda}$ is a $k^*$-g-fusion frame for $H$. Indeed, for each $f\in H$ we have
\begin{align*}
\Vert kf\Vert^4&\leq\big\vert\langle\sum_{j\in\Bbb J}v_j^2\pi_{W_j}\Lambda^*_{j}\tilde{\Lambda_j}\pi_{\tilde{W_j}}f, kf\rangle\big\vert^2\\
&=\big\vert\sum_{j\in\Bbb J}v_j^2\langle\tilde{\Lambda_j}\pi_{\tilde{W_j}}f, \Lambda_j\pi_{W_j}kf\rangle\big\vert^2\\
&\leq\Big(\sum_{j\in\Bbb J}v_j^2\Vert\tilde{\Lambda_j}\pi_{\tilde{W_j}}f\Vert^2\Big)\Big(\sum_{j\in\Bbb J}v_j^2\Vert\Lambda_j\pi_{W_j}kf\Vert^2\Big)\\
&\leq B\Vert kf\Vert^2\sum_{j\in\Bbb J}v_j^2\Vert\tilde{\Lambda_j}\pi_{\tilde{W_j}}f\Vert^2,
\end{align*}
where $B$ is an upper bound of $\Lambda$.
\begin{theorem}
Let $\Lambda$ be a $k$-g-fusion frame for $H$ with bounds $A, B$ and $k$ be closed range. Then $\Big(k^*S_{\Lambda}^{-1}\pi_{S_{\Lambda}(\mathcal{R}(k))}W_j, \Lambda_j\pi_{W_j}\pi_{S_{\Lambda}(\mathcal{R}(k))}(S_{\Lambda}^{-1})^*k, v_j\Big)$ is a $k$-dual of $\Lambda$.
\end{theorem}
\begin{proof}
We know that $S_{\Lambda}^{-1}S_{\Lambda}\vert_{\mathcal{R}(k)}=Id_{\mathcal{R}(k)}$. Then, we have for each $f\in H$,
\begin{align*}
kf&=S_{\Lambda}(S_{\Lambda}^{-1})^*kf\\
&=S_{\Lambda}\pi_{S_{\Lambda}(\mathcal{R}(k))}(S_{\Lambda}^{-1})^*kf\\
&=\sum_{j\in\Bbb J}v_j^2\pi_{W_j}\Lambda^*_j\Lambda_j\pi_{W_j}\pi_{S_{\Lambda}(\mathcal{R}(k))}(S_{\Lambda}^{-1})^*kf\\
&=\sum_{j\in\Bbb J}v_j^2\pi_{W_j}\Lambda^*_j\Lambda_j\pi_{W_j}\pi_{S_{\Lambda}(\mathcal{R}(k))}(S_{\Lambda}^{-1})^*k\pi_{k^*S_{\Lambda}^{-1}\pi_{S_{\Lambda}(\mathcal{R}(k))}W_j}f.
\end{align*}
On the other hand, we obtain by Remark \ref{re1}, for all $f\in H$,
\begin{small}
\begin{align*}
\sum_{j\in\Bbb J}v_j^2&\Vert\Lambda_j\pi_{W_j}\pi_{S_{\Lambda}(\mathcal{R}(k))}(S_{\Lambda}^{-1})^*k\pi_{k^*S_{\Lambda}^{-1}\pi_{S_{\Lambda}(\mathcal{R}(k))}W_j}f\Vert^2\\
&=\Big\langle S_{\Lambda}\big((S_{\Lambda}^{-1})^*k\pi_{k^*S_{\Lambda}^{-1}\pi_{S_{\Lambda}(\mathcal{R}(k))}W_j}f\big), (S_{\Lambda}^{-1})^*k\pi_{k^*S_{\Lambda}^{-1}\pi_{S_{\Lambda}(\mathcal{R}(k))}W_j}f\Big\rangle\\
&=\Big\langle k\pi_{k^*S_{\Lambda}^{-1}\pi_{S_{\Lambda}(\mathcal{R}(k))}W_j}f\big), (S_{\Lambda}^{-1})^*k\pi_{k^*S_{\Lambda}^{-1}\pi_{S_{\Lambda}(\mathcal{R}(k))}W_j}f\Big\rangle\\
&\leq A^{-1}\Vert k^{\dagger}\Vert^2\Vert k\Vert^2\Vert f\Vert^2
\end{align*}
\end{small}
and the proof is completed by Definition \ref{dual}.
\end{proof}
Let $\Lambda$ be a $k$-g-fusion frame for $H$ and $\tilde{\Lambda}$ be a $k$-g-fusion dual of $\Lambda$. Suppose that $\Bbb I$ is a finite subset of $\Bbb J$ and we define
\begin{align}
S_{\Bbb I}f=\sum_{j\in\Bbb I}v_j^2\pi_{W_j}\Lambda^*_j\tilde{\Lambda_j}\pi_{\tilde{W_j}}f , \ \ \ \ \ (\forall f\in H).
\end{align}
It is easy to check that $S_{\Bbb I}\in\mathcal{B}(H)$, positive and
$$S_{\Bbb I}+S_{\Bbb I^c}=k.$$
\begin{theorem}\label{tg1}
Let $f\in H$, then
\begin{eqnarray*}
\sum_{j\in\Bbb I}v_j^2\langle \tilde{\Lambda}_j \pi_{\tilde{W_j}}f,\Lambda_j \pi_{W_j}kf\rangle-\Vert S_{\Bbb I}f\Vert^2 =\sum_{j\in\Bbb I^c}v_j^2\overline{\langle \tilde{\Lambda}_j \pi_{\tilde{W_j}}f,\Lambda_j \pi_{W_j}kf\rangle}-\Vert S_{\Bbb I^c}f\Vert^2.
\end{eqnarray*}
\end{theorem}
\begin{proof}
For each $f\in H$, we have
\begin{align*}
\sum_{j\in\Bbb I}v_j^2\langle \tilde{\Lambda}_j \pi_{\tilde{W_j}}f,\Lambda_j \pi_{W_j}kf\rangle-\Vert S_{\Bbb I}f\Vert^2&=\langle k^*S_{\Bbb I}f,f\rangle-\Vert S_{\Bbb I}f\Vert^2\\
&=\langle k^*S_{\Bbb I}f,f\rangle-\langle S^{\ast}_{\Bbb I} S_{\Bbb I}f,f\rangle\\
&=\langle(k-S_{\Bbb I})^{\ast}S_{\Bbb I}f,f\rangle\\
&=\langle S^{\ast}_{\Bbb I^c}(k-S_{\Bbb I^c})f,f\rangle\\
&=\langle S^{\ast}_{\Bbb I^c}kf,f\rangle-\langle S^{\ast}_{\Bbb I^c}S_{\Bbb I^c}f,f\rangle\\
&=\langle f, k^*S_{\Bbb I^c}f\rangle-\langle S_{\Bbb I^c}f,S_{\Bbb I^c}f\rangle\\
&=\sum_{j\in\Bbb I^c}v_j^2\overline{\langle \tilde{\Lambda}_j \pi_{\tilde{W_j}}f,\Lambda_j \pi_{W_j}kf\rangle}-\Vert S_{\Bbb I^c}f\Vert^2
\end{align*}
and the proof is completed.
\end{proof}
\begin{theorem}\label{ti1}
Let $\Lambda$ be a Parseval $k$-g-fusion frame for $H$. If $\Bbb I\subseteq\Bbb J$ and  $E\subseteq\Bbb I^c$, then for each $f\in H$,
\begin{align*}
\Vert\sum_{j\in \Bbb I\cup E}v_j^2 \pi_{W_j}\Lambda^*_j \Lambda_j \pi_{W_j}f\Vert^2&-\Vert\sum_{j\in\Bbb I^c\setminus E}v_j^2 \pi_{W_j}\Lambda^*_j \Lambda_j \pi_{W_j}f\Vert^2\\
&=\Vert\sum_{j\in \Bbb I}v_j^2 \pi_{W_j}\Lambda^*_j \Lambda_j \pi_{W_j}f\Vert^2-\Vert\sum_{j\in\Bbb I^c}v_j^2 \pi_{W_j}\Lambda^*_j \Lambda_j \pi_{W_j}f\Vert^2\\
& \ \ \ \ \ \ \ \ \ \ \ \ \ \ \ \ \ \ +2\mbox{Re}\Big(\sum_{j\in E}v_j^2\langle\Lambda_j\pi_{W_j}f, \Lambda_j\pi_{W_j}kk^*f\rangle\Big).
\end{align*}
\end{theorem}
\begin{proof}
Let
$$S_{\Lambda, \Bbb I}f:=\sum_{j\in\Bbb I}v_j^2 \pi_{W_j}\Lambda^*_j \Lambda_j \pi_{W_j}f.$$
Therefore, $S_{\Lambda, \Bbb I}+S_{\Lambda, \Bbb I^c}=kk^*$. Hence,
\begin{align*}
S^2_{\Lambda, \Bbb I}-S^2_{\Lambda, \Bbb I^c}&=S^2_{\Lambda, \Bbb I}-(kk^*-S_{\Lambda, \Bbb I})^2\\
&=kk^*S_{\Lambda, \Bbb I}+S_{\Lambda, \Bbb I}kk^*-(kk^*)^2\\
&=kk^*S_{\Lambda, \Bbb I}-S_{\Lambda, \Bbb I^c}kk^*.
\end{align*}
Now, for each $f\in H$ we have
$$\Vert S^2_{\Lambda, \Bbb I}f\Vert^2-\Vert S^2_{\Lambda, \Bbb I^c}\Vert^2=\langle kk^*S_{\Lambda, \Bbb I}f, f\rangle-\langle S_{\Lambda, \Bbb I^c}kk^*f, f\rangle.$$
Consequently, for $\Bbb I\cup E$ instead of $\Bbb I$:
\begin{align*}
&\Vert\sum_{j\in \Bbb I\cup E}v_j^2 \pi_{W_j}\Lambda^*_j \Lambda_j \pi_{W_j}f\Vert^2-\Vert\sum_{j\in\Bbb I^c\setminus E}v_j^2 \pi_{W_j}\Lambda^*_j \Lambda_j \pi_{W_j}f\Vert^2\\
=&\sum_{j\in \Bbb I\cup E}v_j^2\langle\Lambda_j\pi_{W_j}f, \Lambda_j\pi_{W_j}kk^*f\rangle-\sum_{j\in \Bbb I^c\setminus E}v_j^2\overline{\langle\Lambda_j\pi_{W_j}f, \Lambda_j\pi_{W_j}kk^*f\rangle}\\
=&\sum_{j\in \Bbb I}v_j^2\langle\Lambda_j\pi_{W_j}f, \Lambda_j\pi_{W_j}kk^*f\rangle-\sum_{j\in \Bbb I^c}v_j^2\overline{\langle\Lambda_j\pi_{W_j}f, \Lambda_j\pi_{W_j}kk^*f\rangle}\\
&\ \ \ \ \ \ \ \ \ \ \ \ \ \ +2\mbox{Re}\Big(\sum_{j\in E}v_j^2\langle\Lambda_j\pi_{W_j}f, \Lambda_j\pi_{W_j}kk^*f\rangle\Big)\\
=&\Vert\sum_{j\in \Bbb I}v_j^2 \pi_{W_j}\Lambda^*_j \Lambda_j \pi_{W_j}f\Vert^2-\Vert\sum_{j\in\Bbb I^c}v_j^2 \pi_{W_j}\Lambda^*_j \Lambda_j \pi_{W_j}f\Vert^2\\
& \ \ \ \ \ \ \ \ \ \ \ \ \ \ +2\mbox{Re}\Big(\sum_{j\in E}v_j^2\langle\Lambda_j\pi_{W_j}f, \Lambda_j\pi_{W_j}kk^*f\rangle\Big).
\end{align*}
\end{proof}
\begin{theorem}
Let $\Lambda$ be a Parseval $k$-g-fusion frame for $H$. If $\Bbb I\subseteq\Bbb J$ and  $E\subseteq\Bbb I^c$, then for any $f\in H$,
\begin{align*}
&\Vert\sum_{j\in \Bbb I}v_j^2 \pi_{W_j}\Lambda^*_j \Lambda_j \pi_{W_j}f\Vert^2+\mbox{Re}\Big(\sum_{j\in \Bbb I^c}v_j^2\langle\Lambda_j\pi_{W_j}f, \Lambda_j\pi_{W_j}kk^*f\rangle\Big)\\
=&\Vert\sum_{j\in \Bbb I^c}v_j^2 \pi_{W_j}\Lambda^*_j \Lambda_j \pi_{W_j}f\Vert^2+\mbox{Re}\Big(\sum_{j\in \Bbb I}v_j^2\langle\Lambda_j\pi_{W_j}f, \Lambda_j\pi_{W_j}kk^*f\rangle\Big)\geq\frac{3}{4}\Vert kk^*f\Vert^2.
\end{align*}
\end{theorem}
\begin{proof}
In Theorem \ref{ti1}, we showed that
\begin{align*}
S^2_{\Lambda, \Bbb I}-S^2_{\Lambda, \Bbb I^c}=kk^*S_{\Lambda, \Bbb I}-S_{\Lambda, \Bbb I^c}kk^*.
\end{align*}
Therefore,
$$S^2_{\Lambda, \Bbb I}+S^2_{\Lambda, \Bbb I^c}=2\Big(\frac{kk^*}{2}-S_{\Lambda, \Bbb I}\Big)^2+\frac{(kk^*)^2}{2}\geq\frac{(kk^*)^2}{2}.$$
Thus,
\begin{align*}
kk^*S_{\Lambda, \Bbb I}+S^2_{\Lambda, \Bbb I^c}+(kk^*S_{\Lambda, \Bbb I}+S^2_{\Lambda, \Bbb I^c})^*&=kk^*S_{\Lambda, \Bbb I}+S^2_{\Lambda, \Bbb I^c}+S_{\Lambda, \Bbb I}kk^*+S^2_{\Lambda, \Bbb I^c}\\
&=kk^*(S_{\Lambda, \Bbb I}+S_{\Lambda, \Bbb I^c})+S^2_{\Lambda, \Bbb I}+S^2_{\Lambda, \Bbb I^c}\\
&\geq\frac{3}{2}(kk^*)^2.
\end{align*}
Now, we obtain for any $f\in H$,
\begin{align*}
&\Vert\sum_{j\in \Bbb I}v_j^2 \pi_{W_j}\Lambda^*_j \Lambda_j \pi_{W_j}f\Vert^2+\mbox{Re}\Big(\sum_{j\in \Bbb I^c}v_j^2\langle\Lambda_j\pi_{W_j}f, \Lambda_j\pi_{W_j}kk^*f\rangle\Big)\\
=&\Vert\sum_{j\in \Bbb I^c}v_j^2 \pi_{W_j}\Lambda^*_j \Lambda_j \pi_{W_j}f\Vert^2+\mbox{Re}\Big(\sum_{j\in \Bbb I}v_j^2\langle\Lambda_j\pi_{W_j}f, \Lambda_j\pi_{W_j}kk^*f\rangle\Big)\\
&=\frac{1}{2}\big(\langle kk^*S_{\Lambda, \Bbb I}f, f\rangle+\langle S^2_{\Lambda, \Bbb I^c}f, f\rangle+\langle f, kk^*S_{\Lambda, \Bbb I}f\rangle+\langle f, S^2_{\Lambda, \Bbb I^c}f\rangle\big)\\
&\geq\frac{3}{4}\Vert kk^*f\Vert^2.
\end{align*}
\end{proof}
\section{Perturbation of $k$-g-Fusion Frames}
Perturbation of frames has been discussed by Cazassa and Christensen in \cite{caz}. In this section, we present some perturbation of $k$-g-fusion frames and review some results about them.
\begin{lemma}(\cite{caz})\label{G1}
Let U be a Linear operator on a Banach space X and assume that there exist  $\lambda_{1},\lambda_{2}\in[0,1) $ such that
 \begin{align*}
  & \Vert x-Ux\Vert \leq\lambda_{1}\Vert x\Vert+\lambda_{2}\Vert Ux\Vert
\end{align*}
for all $ x\in X$. Then $U$ is bounded and invertible. Moreovere,
 \begin{align*}
 \frac{1-\lambda_{1}}{1+\lambda_{2}}\Vert x\Vert\leq\Vert Ux\Vert\leq\frac{1+\lambda_{1}}{1-\lambda_{2}}\Vert x\Vert
 \end{align*}
 and
 \begin{align*}
 \frac{1-\lambda_{2}}{1+\lambda_{1}}\Vert x\Vert\leq\Vert U^{-1}x\Vert\leq\frac{1+\lambda_{2}}{1-\lambda_{1}}\Vert x\Vert
 \end{align*}
 for all $x\in X$.
\end{lemma}
\begin{theorem}\label{p1}
Let $k\in\mathcal{B}(H)$  be closed range, $\Lambda$ be a $k$-g-fusion frame  for $H$ with bounds $A,B$ and $\lbrace\Theta_{j}\in\mathcal{B}(H, H_j)\rbrace_{j\in\Bbb J}$ be a sequence of operators  such that for any finite subset $\Bbb I\subseteq \Bbb J$ and for each $ f\in H$,
\begin{small}
\begin{align*}
\Vert \sum_{j\in\Bbb I}v_{j}^2\big (\pi_{W_{j}}\Lambda^*_ j\Lambda_ j\pi_{W_{j}}f-&\pi_{W_{j}}\Theta^*_{j}\Theta_{j}\pi_{W_{j}}f\big)\Vert \leq
\lambda_1 \Vert \sum_{j\in\Bbb I}v_{j}^2\pi_{W_{j}}\Lambda^*_j\Lambda_j\pi_{W_{j}} f\Vert\\
& +\lambda_2 \Vert \sum_{j\in\Bbb I}v_{j}^2\pi_{W_{j}}\Theta^*_j\Theta_j\pi_{W_{j}} f\Vert
 +\gamma (\sum_{j\in\Bbb I}v_{j}^2\Vert\Lambda_j\pi_{W_{j}} f\Vert^{2})^{\frac{1}{2}},
\end{align*}
\end{small}
where $ 0\leq \max\lbrace \lambda_1+\frac{\gamma}{\sqrt{A}}, \lambda_2\rbrace<1$, then $\Theta:=(W_j, \Theta_j, v_j)$ is a $k$-g-fusion frame for $H$ with bounds
$$ A\frac{1-(\lambda_1+\frac{\gamma}{\sqrt{A}})}{1+\lambda_2} \ \  and \ \ B\big(\frac{1+\lambda_1+\frac{\gamma}{\sqrt{B}}}{1-\lambda_2}\big)$$
\end{theorem}
\begin{proof}
Assume that  $\Bbb I\subseteq \Bbb J$ is a finite subset and  $f\in H$. We have
\begin{small}
\begin{align*}
\Vert \sum_{j\in\Bbb I}v_{j}^2\pi_{W_{j}}\Theta^*_ j\Theta_ j\pi_{W_{j}}f\Vert &\leq \Vert \sum_{j\in\Bbb I}v_{j}^2\big (\pi_{W_{j}}\Lambda^*_ j\Lambda_ j\pi_{W_{j}}f-\pi_{W_{j}}\Theta^*_{j}\Theta_{j}\pi_{W_{j}}f\big)\Vert\\
&\ \ \ \ \ \ \ \ \ \ \ \ \ \ \ \ \ \ \  +\Vert \sum_{j\in\Bbb I}v_{j}^2 \pi_{W_{j}}\Lambda^*_ j\Lambda_ j\pi_{W_{j}}f\Vert\\
&\leq (1+\lambda_1)\Vert \sum_{j\in\Bbb I}v_{j}^2\pi_{W_{j}}\Lambda^*_ j\Lambda_ j\pi_{W_{j}}f\Vert +\lambda_2\Vert \sum_{j\in\Bbb I }v_{j}^2\pi_{W_{j}}\Theta^*_ j\Theta_ j\pi_{W_{j}}f\Vert\\
& \ \ \ \ \ \ \ \ \ \ \ \ \ \ \ \ \ \ +\gamma(\sum_{j\in\Bbb I}v_{j}^2\Vert\Lambda_j\pi_{W_{j}} f\Vert^{2})^{\frac{1}{2}},
\end{align*}
then,
\begin{align*}
\Vert \sum_{j\in\Bbb I}v_{j}^2\pi_{W_{j}}\Theta^*_ j\Theta_ j\pi_{W_{j}}f\Vert&\leq \frac{1+\lambda_1}{1-\lambda_2}\Vert \sum_{j\in\Bbb I}v_{j}^2\pi_{W_{j}}\Lambda^*_ j\Lambda_ j\pi_{W_{j}}f\Vert +\frac{\gamma}{1-\lambda_2} ( \sum_{j\in\Bbb I}v_{j}^2\Vert\Lambda_j\pi_{w_{j}}f\Vert^{2})^\frac{1}{2}\\
&\leq\Big(\frac{1+\lambda_1}{1-\lambda_2}B+\frac{\gamma}{1-\lambda_2}\sqrt{B}\Big)\Vert f\Vert<\infty.
\end{align*}
\end{small}
Therefore,
$\sum_{j\in\Bbb J} v_{j}^2\pi_{W_{j}}\Theta^*_ j\Theta_ j\pi_{W_{j}}f$
is unconditionally convergent.
It follows that $ \Theta:=(W_{j},\Theta_{j}, v_{j})$ is a g-fusion Bessel sequence for $H$ with bound
$$(\frac{1+\lambda_1}{1-\lambda_2}B+\frac{\gamma}{1-\lambda_2}\sqrt{B}).$$
Thus, we obtain by the hypothesis
\begin{align*}
\Vert S_{\Lambda}f-S_{\Theta}f\Vert\leq\lambda_1\Vert S_{\Lambda}f\Vert + \lambda_2\Vert S_{\Theta}f\Vert +\gamma (\sum_{j\in\Bbb I}v_{j}^2\Vert\Lambda_{j}\pi_{w_{j}}f\Vert)^2)^\frac{1}{2}.
\end{align*}
Therefore, by (\ref{sf1})
\begin{align*}
\Vert f-S_{\Theta}S_{\Lambda}^{-1}f\Vert&\leq\lambda_1\Vert f\Vert+\lambda_2\Vert S_{\Theta}S_{\Lambda}^{-1}f\Vert+\gamma\big(\sum_{j\in\Bbb J}v_J^2\Vert\Lambda_j\pi_{W_j}S_{\Lambda}^{-1}f\Vert^2\big)^{\frac{1}{2}}\\
&\leq\Big(\lambda_1+\frac{\gamma}{\sqrt{A}}\Big)\Vert f\Vert+\lambda_2\Vert S_{\Theta}S_{\Lambda}^{-1}f\Vert.
\end{align*}
Since $0\leq \max\lbrace \lambda_1+\frac{\gamma}{\sqrt{A}}, \lambda_2\rbrace<1$, then by Lemma \ref{G1},  $S_{\Theta}S_{\Lambda}^{-1}$ and consequently $S_{\Theta}$ is invertible and we get
$$\frac{1-\lambda_2}{1+\big(\lambda_1+\frac{\gamma}{\sqrt{A}}\big)}\leq \Vert S_{\Lambda}S_{\Theta}^{-1}\Vert\leq\frac{1+\lambda_2}{1+\big(\lambda_1-\frac{\gamma}{\sqrt{A}}\big)}.$$
So,
\begin{align*}
\Vert S_{\Theta}\Vert&\geq\frac{A}{\Vert S_{\Lambda}S_{\Theta}^{-1}\Vert}\\
&\geq\frac{A-(A\lambda_1+\gamma\sqrt{A})}{1+\lambda_2}\Vert kk^*\Vert
\end{align*}
and by Proposition \ref{pr3}, $\Theta$ is a $k$-g-fusion frame for $H$.
\end{proof}
\begin{corollary}\label{p2}
Let  $\Lambda$ be a $k$-g-fusion frame  for $H$ with bounds $A,B$ and $\lbrace\Theta_{j}\in\mathcal{B}(H, H_j)\rbrace_{j\in\Bbb J}$ be a sequence of operators. If there exists a constant $0<R<A$ such that
\begin{align*}
 \sum_{j\in\Bbb J}v_{j}^2\Vert\pi_{W_{j}}\Lambda^*_ j\Lambda_ j\pi_{W_{j}}f-\pi_{W_{j}}\Theta^*_{j}\Theta_{j}\pi_{W_{j}}f\Vert \leq R\Vert k^*f\Vert
\end{align*}
for all $f\in H$, then $\Theta:=(W_j, \Theta_j, v_j)$ is a g-fusion frame for $H$ with bounds
$$A-R \ \ \ \ and \ \ \ \ \min\big\lbrace B+R\sqrt{\frac{B}{A}}, R\Vert k\Vert+\sqrt{B}\big\rbrace.$$
\end{corollary}
\begin{proof}
It is easy to check that $\sum_{j\in\Bbb
J}v_j^2\pi_{W_{j}}\Theta^*_{j}\Theta_{j}\pi_{W_{j}}f$ is convergent
for any $f\in H$. Thus, we obtain for each $f\in H$,
\begin{align*}
\sum_{j\in\Bbb J}v_j^2\Vert \pi_{W_j}\Lambda^*_j\Lambda_j\pi_{W_j}f-\pi_{W_j}\Theta^*_j\Theta_j\pi_{W_j}f\Vert&\leq R\Vert k^*f\Vert\\
&\leq\frac{R}{\sqrt{A}}\big(\sum_{j\in\Bbb J}v_j^2\Vert\Lambda_j\pi_{W_j}f\Vert^2\big)^{\frac{1}{2}}
\end{align*}
and also
\begin{align*}
\sum_{j\in\Bbb J}v_j^2\Vert \pi_{W_j}\Theta^*_j\Theta_j\pi_{W_j}f\Vert&\leq R\Vert k^*f\Vert+\sum_{j\in\Bbb J}v_j^2\Vert\pi_{W_j}\Lambda^*_j\Lambda_j\pi_{W_j}f\Vert\\
&\leq(R\Vert k\Vert+\sqrt{B})\Vert f\Vert.
\end{align*}
By using Theorem \ref{p1} with  $\lambda_1=\lambda_2=0$ and $\gamma=\frac{R}{\sqrt{A}}$, the proof is completed.
\end{proof}
The proof of the following is similar to the proof of Theorem \ref{p1}.
\begin{theorem}
Let $\Lambda$ be a $k$-g-fusion frame for $H$ with bounds $A,B$ and $\lbrace\Theta_{j}\in\mathcal{B}(H, H_j)\rbrace_{j\in\Bbb J}$ be a sequence of operators  such that for any finite subset $\Bbb I\subseteq \Bbb J$ and for each $f\in H$,
\begin{small}
\begin{align*}
\Vert \sum_{j\in\Bbb I}v_{j}^2\big (\pi_{W_{j}}\Lambda^*_ j\Lambda_ j\pi_{W_{j}}f-&\pi_{W_{j}}\Theta^*_{j}\Theta_{j}\pi_{W_{j}}f\big)\Vert \leq
\lambda_1 \Vert \sum_{j\in\Bbb I}v_{j}^2\pi_{W_{j}}\Lambda^*_j\Lambda_j\pi_{W_{j}} f\Vert\\
& +\lambda_2 \Vert \sum_{j\in\Bbb I}v_{j}^2\pi_{W_{j}}\Theta^*_j\Theta_j\pi_{W_{j}} f\Vert
 +\gamma \Vert k^*f\Vert,
\end{align*}
\end{small}
where $ 0\leq \max\lbrace \lambda_1+\frac{\gamma}{\sqrt{A}\Vert k\Vert}, \lambda_2\rbrace<1$, then $\Theta:=(W_j, \Theta_j, v_j)$ is a $k$-g-fusion frame for $H$ with bounds
$$ A\frac{1-(\lambda_1+\frac{\gamma}{\sqrt{A}}\Vert k\Vert)}{1+\lambda_2} \ \  and \ \ B\big(\frac{1+\lambda_1+\frac{\gamma}{\sqrt{B}}\Vert k\Vert}{1-\lambda_2}\big).$$
\end{theorem}
\begin{theorem}
Let  $\Lambda$ be a $k$-g-fusion frame  for $H$ with bounds $A,B$ and $\lbrace\Theta_{j}\in\mathcal{B}(H, H_j)\rbrace_{j\in\Bbb J}$ be a sequence of operators. If there exists a constant $0<R<A$ such that
\begin{align*}
 \sum_{j\in\Bbb J}v_{j}^2\Vert\Lambda_ j\pi_{W_{j}}f-\Theta_{j}\pi_{W_{j}}f\big)\Vert^2 \leq R\Vert k^*f\Vert^2
\end{align*}
for all $f\in H$, then $\Theta:=(W_j, \Theta_j, v_j)$ is a $k$-g-fusion frame for $H$ with bounds
$$(\sqrt{A}-\sqrt{R})^2 \ \ \ \ and \ \ \ \ (\Vert k\Vert\sqrt{R}+\sqrt{B})^2.$$
\end{theorem}
\begin{proof}
Let $f\in H$. By the triangle and Minkowski inequality, we can write
\begin{small}
\begin{align*}
\Big(\sum_{j\in\Bbb J}v_j^2\Vert\Theta_j\pi_{W_j}f\Vert^2\Big)^{\frac{1}{2}}&\leq
\Big(\sum_{j\in\Bbb J}v_{j}^2\Vert\Lambda_ j\pi_{W_{j}}f-\Theta_{j}\pi_{W_{j}}f\big)\Vert^2\Big)^{\frac{1}{2}}+\Big(\sum_{j\in\Bbb J}v_j^2\Vert\Lambda_j\pi_{W_j}f\Vert^2\Big)^{\frac{1}{2}}\\
&\leq(\Vert k\Vert\sqrt{R}+\sqrt{B})\Vert f\Vert.
\end{align*}
\end{small}
Also
\begin{small}
\begin{align*}
\Big(\sum_{j\in\Bbb J}v_j^2\Vert\Theta_j\pi_{W_j}f\Vert^2\Big)^{\frac{1}{2}}&\geq
\Big(\sum_{j\in\Bbb J}v_j^2\Vert\Lambda_j\pi_{W_j}f\Vert^2\Big)^{\frac{1}{2}}-
\Big(\sum_{j\in\Bbb J}v_{j}^2\Vert\Lambda_ j\pi_{W_{j}}f-\Theta_{j}\pi_{W_{j}}f\big)\Vert^2\Big)^{\frac{1}{2}}\\
&\geq(\sqrt{A}-\sqrt{R})\Vert k^*f\Vert.
\end{align*}
\end{small}
Thus, these complete the proof.
\end{proof}

\end{document}